\documentclass[letterpaper, 10 pt, conference]{ieeeconf}  % Comment this line out if you need a4paper

\IEEEoverridecommandlockouts      % This command is only needed if 
\usepackage[bmargin=25.4mm, rmargin=19.1mm, lmargin=19.1mm, tmargin=19.1mm]{geometry}
\usepackage{graphicx}
\usepackage{xcolor}
\usepackage{dsfont}
\usepackage{subcaption}
\usepackage{mwe}

\usepackage{amsmath} % assumes amsmath package installed
\usepackage{amssymb}  % assumes amsmath package installed
\usepackage{tikz}
\usetikzlibrary{calc,trees,positioning,arrows,chains,shapes.geometric,%
    decorations.pathreplacing,decorations.pathmorphing,shapes,%
    matrix,shapes.symbols}
\usetikzlibrary{quotes,angles}    
\tikzset{
>=stealth',
  punktchain/.style={
    rectangle, 
    rounded corners, 
    draw=black, very thick,
    text width=10em, 
    minimum height=3em, 
    text centered, 
    on chain},
  line/.style={draw, thick, <-},
  element/.style={
    tape,
    top color=white,
    bottom color=blue!50!black!60!,
    minimum width=8em,
    draw=blue!40!black!90, very thick,
    text width=10em, 
    minimum height=3.5em, 
    text centered, 
    on chain},
  every join/.style={->, thick,shorten >=1pt},
  decoration={brace},
  tuborg/.style={decorate},
  tubnode/.style={midway, right=2pt},
}

\usepackage{amsthm}
\newtheorem{theorem}{Theorem}
\newtheorem{definition}{Definition}
\newtheorem{remark}{Remark}
\renewcommand{\qedsymbol}{$\blacksquare$}

\newcommand{\bal}[1] {\ensuremath{\left(\begin{array}{#1}}}
\newcommand{\ear} {\ensuremath{\end{array}\right)}}

\newcommand{\bals}[1] {\ensuremath{\left[\begin{array}{#1}}} % Begin Array Left Square
\newcommand{\ears} {\ensuremath{\end{array} \right] }} % End Array Right Square
\let\leq\leqslant
\let\geq\geqslant

\newcommand{\calE}{\ensuremath{\mathcal{E}}}

\newcommand{\calI}{\ensuremath{\mathcal{I}}}

\newcommand{\calL}{\ensuremath{\mathcal{L}}}

\newcommand{\bmat}{\begin{matrix}}
\newcommand{\emat}{\end{matrix}}
\newcommand{\bbm}{\begin{bmatrix}}
\newcommand{\ebm}{\end{bmatrix}}
\newcommand{\bpm}{\begin{pmatrix}}
\newcommand{\epm}{\end{pmatrix}}
\newcommand{\bse}{\begin{subequations}}
\newcommand{\ese}{\end{subequations}}
\newcommand{\beq}{\begin{equation}}
\newcommand{\eeq}{\end{equation}}
\newcommand{\ben}{\begin{enumerate}}
\newcommand{\een}{\end{enumerate}}
\newcommand{\beni}{\renewcommand{\labelenumi}{\roman{enumi}.}
\renewcommand{\theenumi}{\roman{enumi}}\begin{enumerate}}
\newcommand{\eeni}{\end{enumerate}\renewcommand{\labelenumi}{\arabic{enumi}.}
\renewcommand{\theenumi}{\arabic{enumi}}}
\newcommand{\bena}{\renewcommand{\labelenumi}{\alpha{enumi}.}
\renewcommand{\theenumi}{\alpha{enumi}}\begin{enumerate}}
\newcommand{\eena}{\end{enumerate}\renewcommand{\labelenumi}{\arabic{enumi}.}
\renewcommand{\theenumi}{\arabic{enumi}}}
\newcommand{\bit}{\begin{itemize}}
\newcommand{\eit}{\end{itemize}}

\newcommand{\R}{\ensuremath{\mathbb R}}

\title{\LARGE \bf
Multi-agent formation control for circumnavigation of dynamic shapes}

\author{Joana Fonseca, Jieqiang Wei$^*$, Karl H. Johansson and Tor Arne Johansen% <-this % stops a space
\thanks{*: The corresponding author}
\thanks{
This work is supported by Knut and Alice Wallenberg Foundation, Swedish Research Council, Swedish Foundation for Strategic Research, and Research Council of Norway, CoE AMOS grant number 223254.}
\thanks{Joana Fonseca, Jieqiang Wei, Karl H. Johansson are with the ACCESS Linnaeus Centre, School of Electrical Engineering and Computer Science. 
 KTH Royal Institute of Technology,
 SE-100 44 Stockholm, Sweden.
 {\tt\small \{jfgf, jieqiang, kallej\}@kth.se}.
}
\thanks{Tor Arne Johansen is with Department of Engineering Cybernetics,
Centre for Autonomous Marine Operations
and Systems (NTNU AMOS),
Norwegian University of Science
and Technology,
Trondheim N-7491, Norway.
{\tt\small tor.arne.johansen@ntnu.no}.
}
}

\begin{document}

\maketitle
\thispagestyle{empty}
\pagestyle{empty}

%%%%%%%%%%%%%%%%%%%%%%%%%%%%%%%%%%%%%%%%%%%%%%%%%%%%%%%%%%%%%%%%%%%%%
\begin{abstract}

The problem of multi-agent formation control for target tracking is considered in this paper. The target is an irregular dynamic shape approximated by a circle with moving centre and varying radius.
It is assumed that there are \textit{n} agents and one of them is capable of measuring both the distance to the boundary of the target and to its centre. All the agents must circumnavigate the boundary of the target while forming a regular polygon. We also consider a satellite capable of providing an initial noisy estimate of the target.
A control protocol is designed for all agents and the convergence to the desired state is proved up to a limit bound.
A simulated example is provided to verify the performance of the control protocol designed in this paper.

\end{abstract}

%%%%%%%%%%%%%%%%%%%%%%%%%%%%%%%%%%%%%%%%%%%%%%%%%%%%%%%%%%%%%%%%%%%%%
\section{INTRODUCTION}

The use of unmanned vehicles has allowed higher levels of precision and cost efficiency in many research expeditions \cite{Sivertsen}. It is particularly relevant in challenging or hazardous environments, and if real-time data exchange is required \cite{Lucieer}. 

In this paper we consider a circumnavigation problem of a moving shape on a planar surface. This problem has many applications, especially in ocean science, for instance the tracking of oil spills, algal blooms, plumes, frontal zones as well as toxic clouds. These shapes are detected via satellite and may require persistent tracking for real-time data collection. Using circumnavigation we can obtain data from different areas of the target at the same time allowing to explore the developing of different fronts.

Target tracking and multi-agent formation has a long history \cite{egerstedt2001formation,dimarogonas2008stability,Sun2018,cao2007controlling,oh2015survey,lin2005necessary}. A closely related work \cite{Shames2012} proposed one adaptive protocol to circumnavigate around a moving point, e.g., the fish tracking using AUVs. 

For applications like the one presented in this paper, the target to track is not defined as a moving point but as a time varying irregular shape. We assume this shape may be approximated by a moving circle with varying radius. The available literature does not present a solution to this problem.

There are several methods to solve the monitoring of potential irregular dynamic shapes. One of them would be via satellite. The advantage of this method would be providing an image of the target location, but the disadvantage is the low frequency of measurements, namely satellites that are not geosynchronous can only measure a specific area of the earth periodically with revisit times in the range of hours to days. Also the image resolution may be poor and there could be low visibility due to atmosphere phenomena. Our approach to the problem is relying on satellite images together with autonomous sensing vehicles.

Multi-agent formation includes several agents with similar or differing sensing equipment such as imaging sensors. The agents may be, for instance, drones with different altitudes and fields of view or surface agents. This formation may include a satellite for a wider view to collect different type of data. Having so, we may find a symbiotic relationship between a team of surface agents or drones and a satellite.

The remaining sections of this paper are organised as follows. In Section \ref{s: preli}, some necessary preliminaries are recalled. In Section \ref{s: problem}, the main problem of interest is formulated. 
The main results are presented in Section \ref{s: main}, where the protocol is designed and its proofs of convergence presented.
Some simulations presenting the performance of the proposed algorithm are given in Section \ref{s:simulation}. 
Concluding remarks and future directions come in Section \ref{s: conclusion}.

\textbf{Notations.} The notations used in this paper are fairly standard. $I_n$ is the $n$ dimensional identity matrix. $\|\cdot\|_p$ denotes the $\ell_p$-norm and  the $\ell_2$-norm is denoted simply as $\|\cdot\|$ without a subscript.
We define a rotation matrix $E$ as  
\begin{equation}
E =
\begin{bmatrix}
0 & 1 \\
-1 & 0
\end{bmatrix} .
\end{equation}

%%%%%%%%%%%%%%%%%%%%%%%%%%%%%%%%%%%%%%%%%%%%%%%%%%%%%%%%%%%%%%%%%%%%%
\section{PRELIMINARIES}\label{s: preli}

In this section, we briefly review some results from graph theory and adaptive control that will be used in this paper.

Here we recall some terminologies from graph theory \cite{Bollobas98}. Let $\mathcal{G}=(\mathcal{V},\mathcal{E})$ be a directed graph with 
node set $\mathcal{V}=\{v_1,\ldots,v_n\}$,  edge set $\mathcal{E}\subseteq\mathcal{V}\times\mathcal{V}$,
An edge of $\mathcal{G}$ is denoted by $e_{ij} := (v_i,v_j)$ and we write $\calI=\{1,2,\ldots,n\}$. 
The set of neighbors of node $v_i$ is 
denoted by $N_i := \{v_j\in\mathcal{V}:e_{ji}\in\mathcal{E}\}$.
A directed path from node $v_i$ to node $v_j$ is a chain of edges from $\calE$ 
such that the first edge starts from $v_i$, the last edge ends at  $v_j$ and 
every edge in between starts where the previous edge ends. A directed path is called a directed ring if the initial node and ending node are coincident.

The incidence matrix of a digraph is denoted as $B\in\R^{n\times m}$, with
$B_{ij}=-1$ if the $j^{\text{th}}$ edge is towards vertex
$i$, and equal to $1$ if the $j^{\text{th}}$ edge is originating from
vertex $i$, and $0$ otherwise.
%If for every two nodes $v_i$ and $v_j$ there is a directed path from $v_i$ to $v_j$, then the graph $\calG$ is called %\emph{strongly connected}.
%A subgraph $\calG' = (\calV',\calE',A')$ of $\calG$ is called a \emph{directed 
%spanning tree} for $\calG$ if $\calV' =\calV $, $\calE' \subseteq \calE$, and for every node $v_i\in 
%\calV'$ there is exactly one node $v_j$ such that $e_{ji}\in \calE'$, except for one 
%node, which is called the root of the spanning tree. 
%Furthermore, we call a node $v\in \calV$ a \emph{root} of $\calG$ if there is a directed 
%spanning tree for $\calG$ with $v$ as a root.
%In other words, if $v$ is a root of $\calG$, then there is a directed path from 
%$v$ to every other node in the graph.
% A digraph $\calG$ is called {\it weakly
%connected} if $\calG^o$ is connected, where  $\calG^o$ is the 
%undirected graph obtained from
%$\calG$ by ignoring the orientation of the edges. 

Persistent excitation plays a key role in establishing parameter convergence in adaptive identification \cite{Anderson77,SHIMKIN1987}. 
\begin{definition}\cite{SHIMKIN1987}
The function $f\in\calL_e^2(\R^n)$ is said to be \emph{persistently exciting (p.e.)} if there exist positive constants $\varepsilon_1, T$ such that for all $\tau > 0$, 
\begin{align*}
    \int_{\tau}^{T+\tau} f(t)f(t)^\top dt > \varepsilon_1 I_n.
\end{align*}
$T$ will be termed an excitation period of $f$.
\end{definition}

%%%%%%%%%%%%%%%%%%%%%%%%%%%%%%%%%%%%%%%%%%%%%%%%%%%%%%%%%%%%%%%%%%%%%%
\section{PROBLEM STATEMENT}\label{s: problem}

In this paper, we consider $n$ agents one of which is a sensing agent. We also consider a satellite that provides a noisy starting estimate of the target. All of the $n$ agents are initialized at position $p_i(0), i\in\calI$, which are outside of the shape and form a counterclockwise directed ring on the surface. It is desired that these agents are capable of circumnavigating an irregular shape, i.e., moving along the boundaries, that varies with time. We assume that this shape is approximated by a moving circle, which is specified by
\begin{equation}\label{e:target}
\begin{aligned} 
(\mathbf{c}(t),r(t)) \in \mathbb{R}^3, \\
\end{aligned}
\end{equation}
where $\mathbf{c}(t)= (x_t(t), y_t(t))$ and $r(t)$ are the centre and the radius of the circle, respectively. Then, the satellite would provide an initial estimate $\mathbf{\hat{c}}(0)= (\hat{x}_t(0), \hat{y}_t(0))$ and $\hat{r}(0)$.
Suppose the estimation of the target, denoted as $\hat{c}$, is given, then the counterclockwise angles between the vector $p_{i}-\hat{c}$ and $p_{i+1}-\hat{c}$ is denoted as $\beta_i$ for $i=1,\dots,n-1$, and the angle between $p_{n}-\hat{c}$ and $p_{1}-\hat{c}$ is denoted as $\beta_n$,
\medskip
 \begin{equation}\label{beta}
 \begin{aligned} 
  \beta_{i} = & \angle(p_{i+1}-\hat{c},p_{i}-\hat{c}), \qquad i=1,\dots,n-1 \\
  \beta_n = & \angle(p_{1}-\hat{c},p_{n}-\hat{c}).
 \end{aligned}
 \end{equation}
 
\medskip

Notice that in this case, 
\begin{align}
    \beta_i(0)\geq 0, \quad \textnormal{and} \quad \sum_{i=1}^{n} \beta_i(0) = 2\pi.
 \end{align} 
 
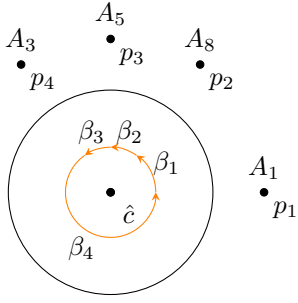
\begin{figure}
\centering
\begin{tikzpicture}
  [
    scale=1.7,
    >=stealth,
    point/.style = {draw, circle,  fill = black, inner sep = 1pt},
    dot/.style   = {draw, circle,  fill = black, inner sep = .2pt},
  ]
  
  % estimated circle
  \def\rad{0.8}
  \node (est) at (0,0) [point, label = {below right:$\hat{c}$}]{};
  \draw (est) circle (\rad);

  % just points on the circle
  \node (c) at (0,0) [point, label] {};
  \node (n1) at (1.2,0) [point, label= $A_1$] {};
  \node (n2) at (0.7,1) [point, label= $A_8$] {};
  \node (n3) at (0,1.2) [point, label= $A_5$] {}; 
  \node (n4) at (-0.7,1) [point, label= $A_3$] {};

  % angle
    \draw[dotted]
    -- (0,0) coordinate (centre) node[left] {}
    -- (1.2,0) coordinate (p1) node[below right] {$p_1$};
    \draw[dotted]
    -- (0,0) coordinate (centre) node[left] {}
    -- (0.7,1) coordinate (p2) node[below right] {$p_2$};
    \draw[dotted]
    -- (0,0) coordinate (centre) node[left] {}
    -- (0,1.2) coordinate (p3) node[below right] {$p_3$};
    \draw[dotted]
    -- (0,0) coordinate (centre) node[left] {}
    -- (-0.7,1) coordinate (p4) node[below right] {$p_4$};

    \draw
    pic["$\beta_1$", draw=orange, ->, angle eccentricity=1.4, angle radius=0.6cm] {angle=n1--c--n2};
    \draw
    pic["$\beta_2$", draw=orange, ->, angle eccentricity=1.4, angle radius=0.6cm] {angle=n2--c--n3};
    \draw
    pic["$\beta_3$", draw=orange, ->, angle eccentricity=1.4, angle radius=0.6cm] {angle=n3--c--n4};
    \draw
    pic["$\beta_4$", draw=orange, ->, angle eccentricity=1.4, angle radius=0.6cm] {angle=n4--c--n1};
    
\end{tikzpicture}   
\caption{Scheme of the system with four agents $A_1$, $A_3$, $A_5$ and $A_8$ at positions $p_1$, $p_4$, $p_3$, $p_2$, respectively. } \label{Scheme1}
\end{figure}

\begin{remark}
In Fig. \ref{Scheme1}, note that agent $A5$ may be in position
$p_3$ and agent $A8$ may be in position $p_2$. So $p_i$ denotes the agent in position $i$. However, we define agent $A1$ to be the sensing agent and it is always in position $p_1$ so agent in position $p_2$ should be the one with the smallest angle to $p_1$, counterclockwise. 
\end{remark}

The kinematic of the agents are of the form
\begin{align} \label{e:dyn_agent}
\mathbf{\dot p_i} = \mathbf{u_i}, \qquad   i\in\calI,
\end{align}
where $\mathbf{p_i}$ is a vector that contains the position $p_i = [ x_i,  y_i ]^{\top}\in\R^2$ 
and $\mathbf{u_i}\in\mathbb{R}^2$ is the control input.

We define the distance to the centre and the boundary of the target circle as 
\begin{equation}\label{distances}
\begin{aligned} 
D^c_i(t) = \|\mathbf{c}(t)-p_i(t)\| \\
D^b_i(t) = |r(t)-D_i(t)|, 
\end{aligned} 
\end{equation}
respectively. Note that even though only one of the agents is able to measure these values, we will use this notation throughout this paper to refer to the distances of each agent to the target, as in Fig. \ref{Scheme}.

\begin{definition}[Circumnavigation]
When the target is stationary, i.e., $\mathbf{c}$ and $r$ are constant, circumnavigation is achieved if the agents   
\begin{enumerate}
\item move in a counterclockwise direction on the boundary of the target, and
\item are equally distributed along the circle, i.e., $\beta_i =\frac{2\pi}{n} $.
\end{enumerate}
More precisely, we say that the circumnavigation is achieved asymptotically if the previous aim is satisfied for $t\rightarrow\infty$. 

For the case with time-varying target, we assume that $\|\dot{\mathbf{c}}\|\leq\varepsilon_1$ and $|\dot{r}|\leq\varepsilon_2$ for some positive constant $\varepsilon_1$ and $\varepsilon_2$.
\end{definition}

\begin{figure}
\centering
\begin{tikzpicture}
  [
    scale=2,
    >=stealth,
    point/.style = {draw, circle,  fill = black, inner sep = 1pt},
    dot/.style   = {draw, circle,  fill = black, inner sep = .2pt},
  ]
 
  % real circle
  \def\rad{0.9}
  \node (real) at (-0.3,0.2) [point, label = {above left:$c$}]{};
  \draw (real) circle (\rad); 
  \node (n5) at (-1.2,0.2) [point, label] {};
  \draw[-] (real) -- node (d) [label = {above left:$r$}] {} (n5);
  
  % estimated circle
  \def\rad{0.8}
  \node (est) at (0,0) [point, label = {below right:$\hat{c}$}]{};
  \draw (est) circle (\rad);

  % just points on the circle
  \node (n1) at (0.7,1) [point, label] {};
  \node (n2) at +(-180:\rad) [point, label] {};
  \node (n3) at (1.2,0) [point, label] {};
  \node (n4) at (0.8,0) [point, label] {};
  
  % triangle edges: connect the vertices, and leave a node at the midpoint
  \draw[-] (est) -- node (b) [label = {below:$\hat{r}$}] {} (n2);
  \draw[-] -- node (c) {} (n3);
  \draw[-] (n3) -- node (n3) [label = {below:$D^b_i$}] {} (n4);
  \draw[-] (est) -- node (n1) [label = {left:$D^c_{i+1}$}] {} (n1);
  
  % angle
    \draw[dotted]
    (0.7,1) coordinate (a) node[right] {$p_{i+1}$}
    -- (0,0) coordinate (b) node[left] {}
    -- (1.2,0) coordinate (c) node[above right] {$p_{i}$};
    
    \draw
    pic["$\beta_i$", draw=orange, ->, angle eccentricity=1.4, angle radius=0.6cm] {angle=c--b--a};

  % uncertain circle    
    \draw[decoration={random steps, amplitude=2mm}, decorate] (-0.3,0.2) circle (0.9);
    
\end{tikzpicture}   
\caption{Scheme of the estimated $\hat{c}$, $\hat{r}$ and the real target $c$, $r$ as well as the angle $\beta_i$ between two vehicles $p_{i+1}$ and $p_i$ } \label{Scheme}
\end{figure}
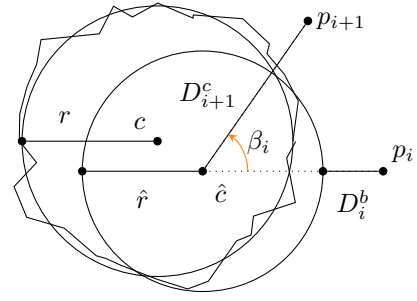

Now we are ready to pose the problem of interest that will be solved in the following sections.

\medskip

\textbf{Problem 1.} 
Design estimator for $\mathbf{c}(t)$ and $r(t)$ when distance measures \eqref{distances} are available to one of the agents, and  design the control input $\mathbf{u}_i$ for all the agents such that for some positive $\varepsilon_1$, $\varepsilon_2$,
\begin{align}
    &\|\dot c \|\leq\varepsilon_1\\
    &|\dot r |\leq\varepsilon_2,
\end{align}
there exist positive $K_1$, $K_2$ and $K_3$ satisfying
\begin{align}
    &\lim\sup_{t\rightarrow\infty} \|\hat{\mathbf{c}}(t)-\mathbf{c}(t)\| \leq K_1 \varepsilon_1,\\
    &\lim\sup_{t\rightarrow\infty} |\hat{r}(t)-r(t)| \leq K_2 \varepsilon_2,\\
    &\lim\sup_{t\rightarrow\infty} |\hat D_i^c-\hat r(t)| \leq K_3 \varepsilon_2,\\
    &\lim_{t\rightarrow\infty} \beta_i = \frac{2\pi}{n}.
\end{align}

Throughout the rest of the paper, we denote the estimate of the target as
\begin{equation}
\begin{aligned} 
(\mathbf{\hat c},\hat r) \in \mathbb{R}^3, \\
\end{aligned}
\end{equation}
where $\mathbf{\hat c}$ = ($\hat x_t$ $\hat y_t$).

%%%%%%%%%%%%%%%%%%%%%%%%%%%%%%%%%%%%%%%%%%%%%%%%%%%%%%%%%%%%%%%%%%%
\section{MAIN RESULTS}\label{s: main}

Here we present our solution for Problem 1. We consider $n$ agents at positions $p_i(t)$ and we assume only one of them is capable of measuring its distance $D_i^b(t)$ to the target boundary as well as its distance $D_i^c(t)$ to the target centre. Then, it should estimate $(\mathbf{c}(t),r(t))$ from its distance measures $D_i^b(t)$, $D_i^c(t)$ and share the information with the other agents. Each agent calculates its desired velocity taking into account its angle $\beta_{i}(t)$ to the next agent as well as its distance to the target centre and boundary, obtained with the estimates of the target. The scheme on Fig.\ref{Scheme5} summaries this algorithm loop.

Motivated by \cite{Shames2012}, we propose the following adaptive estimation of the radius $r(t)$ of the target using the sensing agent $A1$ in position $p_1$.
Observe that
\begin{align}\label{e:d-D^b}
\frac{d}{dt}(D_1^b)^2 = 2(\dot{r}-\dot{D_1^c})(r-D_1^c).
\end{align}
Assume the estimate of $r$ is denoted as $\hat{r}$, we have
\begin{align}
\frac{1}{2}\big( \frac{d}{dt}(D_1^b)^2-\frac{d}{dt} (D_1^c)^2 \big) +\dot{D_1^c}\hat{r} = \dot{D_1^c}(\hat{r}-r) + \dot{r}(r-D_1^c).
\end{align}
Then for some positive constant $\gamma$ the dynamic
\begin{align}\label{e:esti-c-not-applicable}
\dot{\hat{r}} = -\gamma \dot{D_1^c} \big[ \frac{1}{2}\big( \frac{d}{dt}(D_1^b)^2-\frac{d}{dt} (D_1^c)^2 \big) +\dot{D_1^c}\hat{r} \big]
\end{align}
can estimate the variable $r$ under the persistent excitation condition. Indeed, in this case
\begin{align}
\frac{d}{dt} (\hat{r}-r) =  - \gamma (\dot{D_1^c})^2 (\hat{r}-r) - \vartheta_{\dot{r}},
\end{align} 
where $\vartheta_{\dot{r}} = \dot{r}(\gamma\dot{D_1^c}(r-D_1^c)+1)$ is bounded by $M_1\varepsilon_2$. Indeed all its elements are bounded by $M_1$ and recall that $|\dot{r}|\leq\varepsilon_2$. Note that $r-D_1^c$ is bounded because $r$ and $D_1^c$ are bounded as well. Furthermore, as it will be clear soon, $\vartheta_{\dot{r}}$ can be replaced by $\vartheta_{\dot{r}} = \dot{r}(\gamma V(r-D_1^c)+1)$ using equations \eqref{num1} and \eqref{num2}, where $V$ is the bounded estimate of $\dot{D_1^c}$.

However, the implementation of \eqref{e:esti-c-not-applicable} needs the derivative of $D_1^b$ and $D_1^c$ which is not desired. It would require explicit differentiation of
measured signals with accompanying noise amplification. Thus, for some positive constant $\alpha$ we adopt the state variable filtering and then design the estimator as follows
\begin{align}\label{e:aug-filtering}
\dot{z}_1 & = -\alpha z_1(t) + \frac{1}{2}(D_1^b)^2 \\
\eta(t) & =  \dot{z}_1 \\
\dot{z}_2 & = -\alpha z_2(t) + \frac{1}{2}(D_1^c)^2 \\ \label{e:reused}
m(t) & = \dot{z}_2 \\
\dot{z}_3 & = -\alpha z_3(t) + D_1^c \label{num1}\\ 
V(t) & = \dot{z}_3 \label{num2}
\end{align}
with initial conditions $z_1(0)=z_2(0)=z_3(0)=0.$ Now together the above dynamics, the estimator for $r$ is given as
\begin{align}  \label{estimateradius}
\dot{\hat{r}} = -\gamma V \big[ \eta-m +V \hat{r} \big]. 
\end{align} 

Now we need to know $\mathbf{c}(t)$ but we only know $D_1^c(t)$ and $D_1^b(t)$. Thus, we must use again adaptive estimation for the centre $\mathbf{c}(t)$ of the target.

Observe that
\begin{align}
\frac{d}{dt}(D_1^c)^2 = 2(\dot{p_1}-\mathbf{\dot{c}})^\top(p_1-\mathbf{c}).
\end{align}
Assume the estimation of ${\mathbf{c}}$ is denoted as ${\mathbf{\hat{c}}}$, we have
\begin{align}
\frac{1}{2}\big( \frac{d}{dt}(D_1^c)^2-\frac{d}{dt} {\|p_1\|}^2 \big) +\dot{p_1}^\top\mathbf{\hat{c}}& \nonumber \\ 
 = \dot{p_1}^\top(\mathbf{\hat{c}}-\mathbf{c}) & + \mathbf{\dot{c}}^\top(\mathbf{c}-p_1).
\end{align}
Then the dynamic
\begin{align}\label{e:esti-r-not-applicable}
\mathbf{\dot{\hat{c}}} = -\gamma \dot{p_1} \big[ \frac{1}{2}\big( \frac{d}{dt}(D_1^c)^2-\frac{d}{dt} {\|p_1\|}^2 \big) +\dot{p_1}^\top\mathbf{\hat{c}} \big]
\end{align}
can estimate the parameter $\mathbf{c}$ under some persistent excitation condition on $\dot{p_1}$. Indeed, in this case
\begin{align}
\frac{d}{dt} (\mathbf{\hat{c}} - \mathbf{c}) =  - \gamma {\|\dot{p_1}\|}^2 (\mathbf{\hat{c}} - \mathbf{c}) - \vartheta_{\dot{\mathbf{c}}},
\end{align} 
where $\vartheta_{\dot{\mathbf{c}}} = \gamma\dot{\mathbf{c}}^\top\dot{p_1}(\mathbf{c}-p_1)+\dot{\mathbf{c}}$ is bounded by $M_2\varepsilon_1$. Indeed all its elements are bounded by $M_2$ and recall that $|\dot{\mathbf{c}}|\leq\varepsilon_1$. Note that $\mathbf{c}-p_1$ is bounded because  $\mathbf{c}$ and $p_1$ are within a finite map. Furthermore, as it will be clear soon, $\vartheta_{\dot{\mathbf{c}}}$ can be replaced by $\vartheta_{\dot{\mathbf{c}}} = \gamma\dot{\mathbf{c}}^\top V_2(\mathbf{c}-p_1)+\dot{\mathbf{c}}$ using equations \eqref{num3}-\eqref{num4}, where $V_2$  is the estimate of $\dot{p_1}$ and it is bounded.

However, the implementation of \eqref{e:esti-r-not-applicable} needs the derivative of $p_1(t)$ and $D_1^c(t)$ which is not desired. Therefore we use the previously defined equation \eqref{e:reused} for $D_1^c(t)$ and redefine it as $\eta_2(t)  =  \dot{z}_2$  and add the following
\begin{align}\label{e:aug-filtering2}
\dot{z}_4 & = -\alpha z_4(t) + \frac{1}{2}p_1(t)p_1^T(t) \\
m_2(t) & = \dot{z}_4 \\
\dot{z}_5 & = -\alpha z_5(t) + p_1(t) \label{num3}\\
V_2(t) & = \dot{z}_5 \label{num4}
\end{align}
with initial conditions $z_4(0)=z_5(0)=0.$ Now together the above dynamics, the estimator for \textbf{c} is given as
\begin{align}  \label{estimatecentre}
\mathbf{\dot{\hat{c}}} = -\gamma V_2 \big[ \eta_2-m_2 +V_2^T \mathbf{\hat{c}} \big]. 
\end{align} 

Now, we want to obtain the desired control input $\mathbf{u_i}(t)$ using the previously measured and estimated variables. The total velocity of each agent comprises of two sub-tasks: approaching the target and circumnavigating it. Therefore we define the direction of each agents towards the centre of the target as the bearing $\psi_i(t)$,
\begin{equation} \label{bearing}
 \psi_i(t) = \frac{\mathbf{\hat{c}}(t) - p_i(t)}{\hat{D^c_i}(t)} = \frac{\mathbf{\hat{c}}(t) - p_i(t)}{\|\mathbf{\hat{c}}(t)-p_i(t)\| \\}. 
\end{equation}

Note that $\psi_i$ in \eqref{bearing} is not well-defined when $\hat{D^c_i}=0$, thus we will prove that this singularity is avoided for all time $t\geq 0$ in the third proof of Theorem \ref{Theo1}..

The first sub-task is related to the bearing $\psi_i(t)$ and the second one is related to its perpendicular, $E\psi_i(t)$.
% The weight  for  the  approaching  sub-task  should  be  equivalent to  how  distant  the  agent  is  to  the  boundary  of  the  target and the weight for the circumnavigating sub-task should be equivalent to how distant from the next agent it is.
Therefore, the control law for each agent $i$ is
\begin{equation}\label{e:control-adapt}
\begin{aligned} 
\mathbf{u_i} = & 
\mathbf{\dot{\hat{c}}} + ((\hat{D^c_i} - \hat r) - \dot{\hat{r}})\psi_i + \beta_i \hat{D^c_i} E \psi_i
\end{aligned}
\end{equation}

\begin{remark}
    Note that for implementation we would define $U_i$ as the control input for each agent $i$. Then, $U_i$ must have a maximum absolute value $u_{max}$ since the maximum velocity of the agent would be limited as well. $U_i$ could either be represented as $U_i=\delta u_i$, being $\delta$ some positive parameter for tunning. Or represented as the saturation function: if $\|u_i\| > u_{max}$ then $U_i=\frac{u_{max}}{\|u_i\|}u_i$, else $U_i = u_i$.
\end{remark}

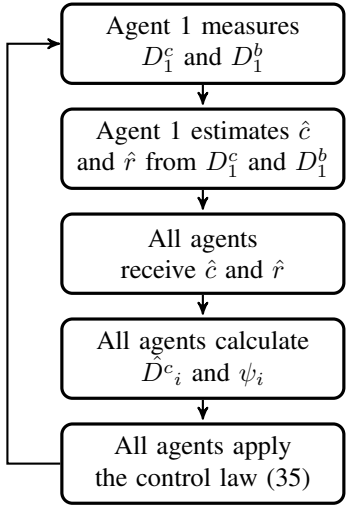
\begin{figure}
\centering
    \begin{tikzpicture}
  [node distance=.3cm, start chain=going below]
     \node[punktchain, join] (1) {Agent 1 measures $D^c_1$ and $D_1^b$};
     \node[punktchain, join] (2) {Agent 1 estimates $\hat{c}$ and $\hat{r}$ from $D^c_1$ and $D_1^b$};
     \node[punktchain, join] (3) {All agents receive $\hat{c}$ and $\hat{r}$};
     \node[punktchain, join] (4) {All agents calculate $\hat{D^c}_i$ and $\psi_i$};
     \node[punktchain, join] (5) {All agents apply the control law \eqref{e:control-adapt}};
     
     \draw[|-,-|,->, thick,] (5.west) -|+(-2em,0)|- (1.west);
  \end{tikzpicture}
\caption{Scheme of the algorithm run on the system} \label{Scheme5}
\end{figure}

\begin{theorem} \label{Theo1}
The initial condition satisfies $\hat{D^c_i}(0)>\hat{r}(0) > 0$. Suppose $\dot{p_1}(t)$ and $\dot{D_1^c}(t)$ are p.e., $\|\dot{\mathbf{c}}\|\leq\varepsilon_1$, and $|\dot{r}|\leq\varepsilon_2$. Consider the system \eqref{e:dyn_agent} with the control protocol \eqref{e:control-adapt}, and the initialisation satisfying  $\|p_i(0)-\hat{c}(0)\|>0$, then there exists $K_1$, $K_2$ and $K_3$ such that circumnavigation of the moving circle with equally spaced agents can be achieved asymptotically up to a bounded error, i.e.
\begin{align}\label{oi}
    &\lim\sup_{t\rightarrow\infty} \|\hat{\mathbf{c}}(t)-\mathbf{c}(t)\| \leq K_1 \varepsilon_1,\\\label{oi2}
    &\lim\sup_{t\rightarrow\infty} |\hat{r}(t)-r(t)| \leq K_2 \varepsilon_2,\\\label{oi3}
    &\lim\sup_{t\rightarrow\infty} |\hat D_i^c-\hat r(t)| \leq K_3 \varepsilon_2,\\ \label{oi4}
    &\lim_{t\rightarrow\infty} \beta_i = \frac{2\pi}{n}.
\end{align}
\end{theorem}

\begin{proof}
The proof is divided into four parts. In the first part, we prove that \eqref{oi} and \eqref{oi2} hold. In the second part, we prove that the estimated distance $\hat{D^c_i}$ converges to the estimated radius $\hat{r}$, or in other words, that \eqref{oi3} holds. In the third part we prove that the singularity of the bearing $\psi_i(t)$ is avoided. In the last part, we show that the angle between the agents will converge to the average consensus for $n$ agents, $\beta_i = \frac{2\pi}{n}$, meaning \eqref{oi4} holds. We will assume the implementable controller is given by $U_i=\delta u_i$.
\begin{enumerate}

    \item Firstly, we prove that \eqref{oi} and \eqref{oi2} hold. The proof for boundedness of the centre \eqref{oi}, can be found on \cite{Shames2012}, Proposition 7.1. The proof for boundedness of the radius, however, needs to be derived in this paper.
    Then, we have that 
    \begin{equation}
    \begin{aligned}
        \dot{\tilde{r}} &= \dot{\hat{r}} = -\gamma V \big[ \eta-m +V \hat{r} \big] \\
        &= -\gamma V \big[ \eta-m +V (\tilde{r}+r) \big]\\
        &= -\gamma V^2\tilde{r} - \gamma V\big[ \eta-m +V r \big]\\
        &= -\gamma V^2\tilde{r} + G(t) \\
    \end{aligned}
    \end{equation}
    where $G(t) = - \gamma V\big[ \eta-m +V r \big]$. We know that $|G(t)|\leq k_1\epsilon_2$ for some $k_1,\epsilon_2 \geq 0$ because V is bounded and that $|\eta-m +Vr| < k_2$ we can prove that for a Lyapunov function $W_r=\frac{1}{2}\tilde r^2$ we get
    \begin{equation}
    \begin{aligned}
        \dot W_r &= \tilde{r} \dot{\tilde{r}} = \tilde{r} (-\gamma V^2\tilde{r} +G(t)) \\
        &=  -\gamma V^2\tilde{r}^2 +\tilde{r}G(t)\\
        &\leq -\gamma V^2\tilde{r}^2 +k_1\epsilon_2\tilde{r}
    \end{aligned}
    \end{equation}
    then we get that for $\dot W_r \leq 0$ to hold, $-\gamma V^2\tilde{r}^2 +k_1\epsilon_2\tilde{r}\leq 0 $ must hold. So, we have that when $\tilde{r} \geq \frac{k_1\epsilon_2}{\gamma V^2}$ or $\tilde{r} \leq -\frac{k_1\epsilon_2}{\gamma V^2}$, $\dot W_r \leq 0$ so that $|\tilde{r}|$ is within $\pm\frac{k_1\epsilon_2}{\gamma V^2}$. This error $\tilde{r}$ is then proved to converge asymptotically to a ball since $\dot{D^c_1}$ is p.e.. 
    
    \item We prove that all agents reach the estimate of the boundary of the moving circles asymptotically, i.e., $\lim_{t\rightarrow\infty}{\|p_i(t)-\hat{c}(t)\|} = \lim_{t\rightarrow\infty}{\hat{D^c_i}(t)} = \hat r(t)$, so \eqref{oi3}  holds.     
        
    Consider the function $W_i(t):=\hat{D^c_i}(t)-\hat r(t)$ whose time derivative for $t\in[0,\tau_{\max})$ is given as
    \begin{align*}
        \dot W_i = & \frac{(\mathbf{\hat{c}}-p_i)^\top (\mathbf{\dot{\hat{c}}}-\dot{p}_i)}{\hat{D^c_i}} - \dot{\hat{r}} \\
         =  & -\frac{(\mathbf{\hat{c}}-p_i)^\top}{\hat{D^c_i}} \delta ((\hat{D^c_i} - \hat r - \dot{\hat{r}})\psi_i + \beta_i \hat{D^c_i} E\psi_i)\\
         & \qquad - \dot{\hat{r}}\\
         =  & -\frac{(\mathbf{\hat{c}}-p_i)^\top}{\hat{D^c_i}}\psi_i\delta (\hat{D^c_i} - \hat r - \dot{\hat{r}}) \\
         & \qquad -\frac{(\mathbf{c}-p_i)^\top}{\hat{D^c_i}}E\psi_i\delta \beta_i \hat{D^c_i} - \dot{\hat{r}} \\
         = & - \delta(\hat{D^c_i} - \hat{r} - \dot{\hat{r}}) - \dot{\hat{r}} \\
         = & - \delta W_i.
    \end{align*}
    Hence for $t\in[0,+\infty)$, we have $\hat{D^c_i}(t) = \delta W_i(0) e^{-t} + \hat{r}(t)$ which implies $W_i$ is converging to zero exponentially.
    
    \item Now, we prove that $\psi_i$ in \eqref{bearing} is well-defined, or in  other words, that its singularity is avoided for all time $t\geq 0$, $\hat{D^c_i}\ne0$ $\forall t$.

    Having $\hat{D^c_i}(t) = \delta W_i(0) e^{-t} + \hat{r}(t)$ from the previous proof and knowing that $W_i(0)$ is always positive and that it converges to zero exponentially, we have that if $\hat{r}(t)>0$ then $\hat{D^c_i}(t)>0$, $\forall t$.
    
    So we want to prove that $\hat{r}(t)>0$ $\forall t$. We can get that $\hat{r} = r+\tilde{r} \Leftrightarrow r - \frac{k_1\epsilon_2}{\gamma V^2} \leq \hat{r} \leq r + \frac{k_1\epsilon_2}{\gamma V^2}$. Then, since we know $\frac{k_1\epsilon_2}{\gamma V^2}$ is a small bound, we get that $\hat{r} \geq r - \frac{k_1\epsilon_2}{\gamma V^2} > 0$. 
    
    Then we conclude that $\hat{D^c_i}\ne0$ $\forall t$ and that the bearing $\psi_i(t)$ is well defined $\forall t$.
    
    \item Finally, we show that the angle between the agents will converge to the average consensus for $n$ agents, $\beta_i = \frac{2\pi}{n}$, so \eqref{oi4} holds. 
    
    Firstly, note that we can write an angle between two vectors $\beta_i=\angle(v_2,v_1)$ as
    \begin{equation}
        \beta_i=2atan2((v_1\times v_2)\cdot z,\|v_1\|\|v_2\|+v_1\cdot v_2)
    \end{equation}
    and its derivative as
    \begin{equation}
        \dot\beta_i=\frac{\hat{v_1}\times z}{\|v_1\|}\dot{v_1}-\frac{\hat{v_2}\times z}{\|v_2\|}\dot{v_2}
    \end{equation}
    where $z = \frac{v_1\times v_2}{\|v_1\times v_2\|}, \hat{v_i} = \frac{v_1}{\|v_i\|}, i=1,2$.
    
    Then, for $v_1=p_i-\hat{c}$ and $v_2=p_{i+1}-\hat{c}$ we get
    \begin{equation*}
    \begin{aligned}
        \dot{\beta_i} &= \frac{\hat{v_1}\times z}{\|v_1\|}\dot{v_1}-\frac{\hat{v_2}\times z}{\|v_2\|}\dot{v_2} \\
        &= \frac{\hat{v_1}\times z}{\|v_1\|}\delta((\hat{D^c_i} - \hat r - \dot{\hat{r}})\psi_i + \beta_i \hat{D^c_i} E \psi_i) \\
        & \qquad -\frac{\hat{v_2}\times z}{\|v_2\|}\delta((\hat{D}^c_{i+1} - \hat r - \dot{\hat{r}})\psi_{i+1} \\
        & \qquad \qquad + \beta_{i+1} \hat{D^c_{i+1}}  E \psi_{i+1}) \\
        % &= -\frac{1}{\|v_1\|} \beta_{i} + \frac{1}{\|v_2\|} \beta_{i+1} \\
        &= \delta (-\beta_{i} + \beta_{i+1}), \qquad i= 1,\ldots,n-1 \\ 
        \dot{\beta}_n &= \delta (-\beta_{n} + \beta_{1}).
    \end{aligned}
    \end{equation*}
    which can be written in a compact form as following
    \begin{align}
        \dot{\beta} = - \delta B^\top \beta \label{e:dyn-beta}
    \end{align}
    where $B$ is the incidence matrix of the directed ring graph from $v_1$ to $v_n$. 
    
    First, we note that the system \eqref{e:dyn-beta} is positive (see e.g., \cite{farina2000positive}), i.e., $\beta_i(t)\geq 0$ if $\beta_i(0)\geq 0$ for all $t\geq 0$ and $i\in\calI$. This proves the positions of the agents are not interchangeable.
    
    Second, noticing that $B^\top$ is the (in-degree) Laplacian of the directed ring graph which is strongly connected, then by Theorem 6 in \cite{Wei2018}, we  $\beta$ converges to consensus $\frac{2\pi}{n}\mathds{1}$.
    
%    Note that $\hat{r}(t)$ converges a neighbourhood of $r(t)\in[\underline{r}, \overline{r}]$ with $\underline{r}$, we can assume without loss of generality $\hat{r}(t)\in[\underline{\hat{r}}, \overline{\hat{r}}]$ with $\underline{\hat{r}}>0$.   
%    Since $\hat{D}^c_i(t)$ converges to $\hat{r}(t)$ as $t\rightarrow\infty$, then for any $\varepsilon_3 > 0$, there exists $T$ such that $|\frac{1}{\hat{D}^c_i(t)}-\frac{1}{\hat{r}(t)}|< \varepsilon_3$ for any $t\geq T$ and $i\in\calI$. Then for any $t\geq T$, 
%    \begin{align*}
 %       \dot{W} & \leq  - \frac{1}{\hat{r}(t)} \beta^\top B^\top \beta + \JW{to be fixed}.
%    \end{align*}
    
    \qedsymbol
\end{enumerate}
\end{proof}

\begin{remark}
    Recall remark 1 on agent $A_j$ being in position $p_i$. After the proof note how the agent $A_j$ will necessarily maintain its relative position $p_i$ throughout the circumnavigation mission. Indeed, without loss of generality we could say that agent $A_j$ is always in position $p_i$.
\end{remark}

\begin{remark}
Recall Definition 1 on persistent excitation.
This means that for the persistently exciting condition to apply, the agent must move in a trajectory that is not confined to a straight line in the 2D space. As referred in \cite{Shames2012}, 
"The agent cannot simply head straight towards the target but must execute a richer class of motion.". Then, for a circumnavigation mission we could infer that the p.e. condition is guaranteed. 
\end{remark}

\begin{remark}
The dynamic of \eqref{e:dyn-beta} can be considered as the time-varying of advection systems on graphs \cite{Chapman2011}.
\end{remark}

%%%%%%%%%%%%%%%%%%%%%%%%%%%%%%%%%%%%%%%%%%%%%%%%%%%%%%%%%%%%%%%%%%%%%
\section{SIMULATIONS}\label{s:simulation}

In this section, we present simulations of both protocols designed in section \ref{s: main}. For the first result we used the derived method for estimation of the target \eqref{estimateradius} and \eqref{estimatecentre} and the controlling protocol for the agents \eqref{e:control-adapt}. We simulate a moving target with initial position $(x(0),y(0))=(25,25)$, radius $r(0)=10$ and dynamic according to

\begin{equation} \label{eq1}
\begin{aligned}
    & \dot x(t) = \alpha_1(t) + 0.5 \\    
    & \dot y(t) = \alpha_2(t) + 0.5 \\     
    & \dot r(t) = \alpha_3(t)  
\end{aligned}
\end{equation}

However, we simulate that the satellite will provide as an initial noisy estimate of $(\hat{x}(0),\hat{y}(0))=(25,25)$, radius $\hat{r}(0)=20$. Note that at time $t=0$ the radius estimate is double the real radius.

Here, $\alpha_i(t)$ is a random scalar drawn from the standard normal distribution for $i=1,2,3$. For this generated target we got the following results.
We can see the agents circumnavigating the moving target in Fig.\ref{pathpic} and Fig.\ref{path}. This gives us a more practical idea of how the agents behave in their target-tracking mission. Through their paths we can infer how the target behaved - varying radius and moving centre. 

\begin{figure}
  \begin{subfigure}[t]{.24\textwidth}
    \centering
    \includegraphics[width=\linewidth]{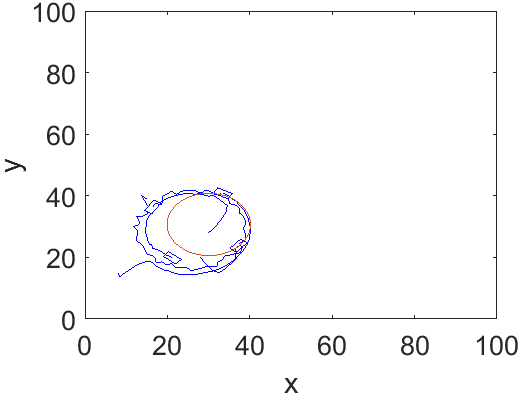}
  \end{subfigure}
  \begin{subfigure}[t]{.24\textwidth}
    \centering
    \includegraphics[width=\linewidth]{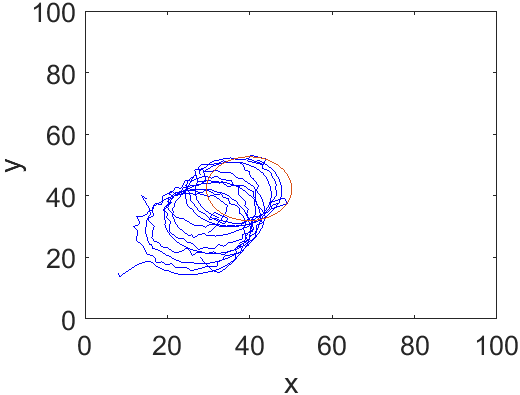}
  \end{subfigure}
  \begin{subfigure}[t]{.24\textwidth}
    \centering
    \includegraphics[width=\linewidth]{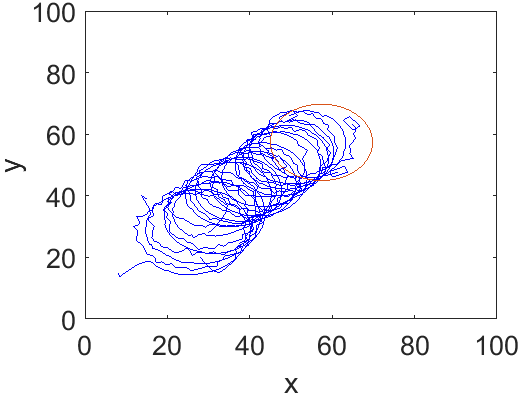}
  \end{subfigure}
  \begin{subfigure}[t]{.24\textwidth}
    \centering
    \includegraphics[width=\linewidth]{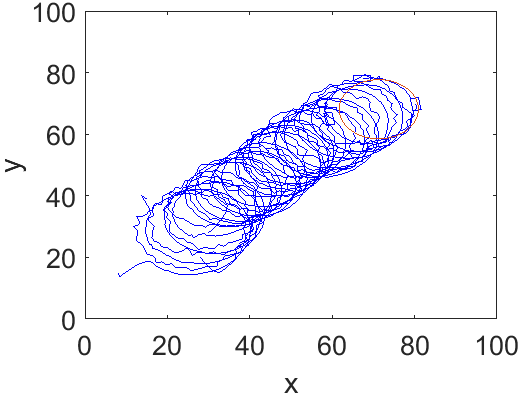}
  \end{subfigure}
  \caption{Time-lapse of four agents (blue rectangles) circumnavigating a moving target (red) with representation of their paths (blue)}
  \label{pathpic}
\end{figure}

\begin{figure}
    \centering
    \includegraphics[width=8.6cm,trim={0 0 0 0},clip]{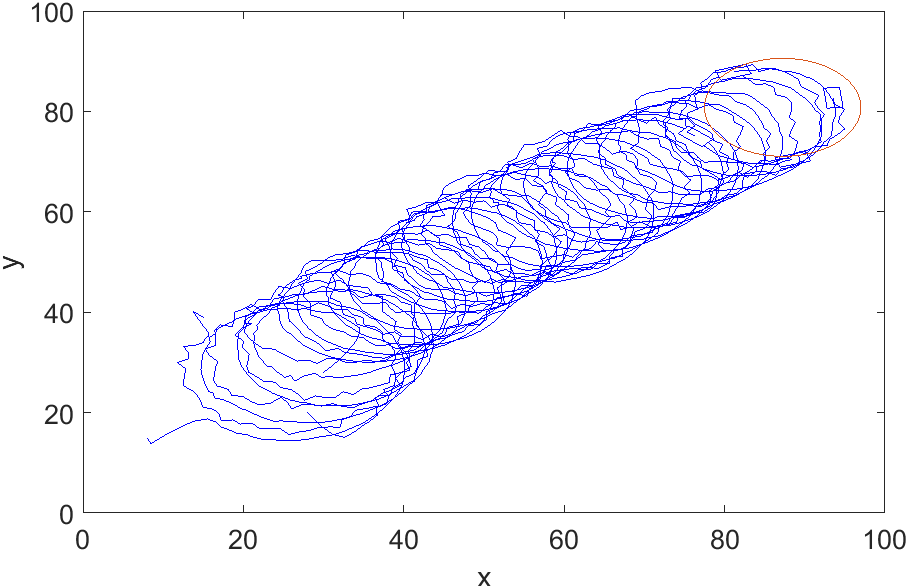}
    \caption{Four agents (blue rectangles) circumnavigating a moving target (red) with representation of their paths (blue)}
    \label{path}
\end{figure}
    
In Fig.\ref{targets} we have plenty of plots. On the first and second row we compare the real and estimated target. Note that the estimate of the centre $\hat{c}(\hat{x},\hat{y})$ has an estimation error of up to 2 units. Also note that the estimate of the radius $\hat{r}$ is composed of two instances. In the first, the initial estimate provided by the satellite was very noisy and so we can see the estimate converging rapidly to a more accurate estimation. In the second we can see an estimation error of up to 2 units.

On the third row left column, we can see the distance $D_i^b$ of each target to the boundary of the target - the perfect tracking would result in a distance $D_i^b$ of 0 for all agents, for every time step. Here we have an error of up to 0.5 units, except for the very beginning where the error can reach 10 units. This is merely because in the beginning the agents are far away from the target. 

On the third row right column, we have the angle between agent A1 and A2, $\beta_1$. Having 4 agents, the perfect tracking would result in $2\pi/4 = \pi/2 \approx 1.57$ for all agents, for every time step. We can see this reference as the red line in the plot so we see that, for agent A1, the error is up to 0.2 radians. 

Finally, on the fourth row we have the control input of agent A1, both in $x$ and $y$ in blue. Recall Remark. 2 where we stated that, for a practical implementation, there should be a maximum velocity $u_{max}$. For this case study we defined that $u_{max}=1.5$ and we plotted this limit in red. Note how the control input stays within the limit values 1.5 and -1.5.

\begin{figure}
    \centering
    \includegraphics[width=8.6cm,trim={0 0 0 0},clip]{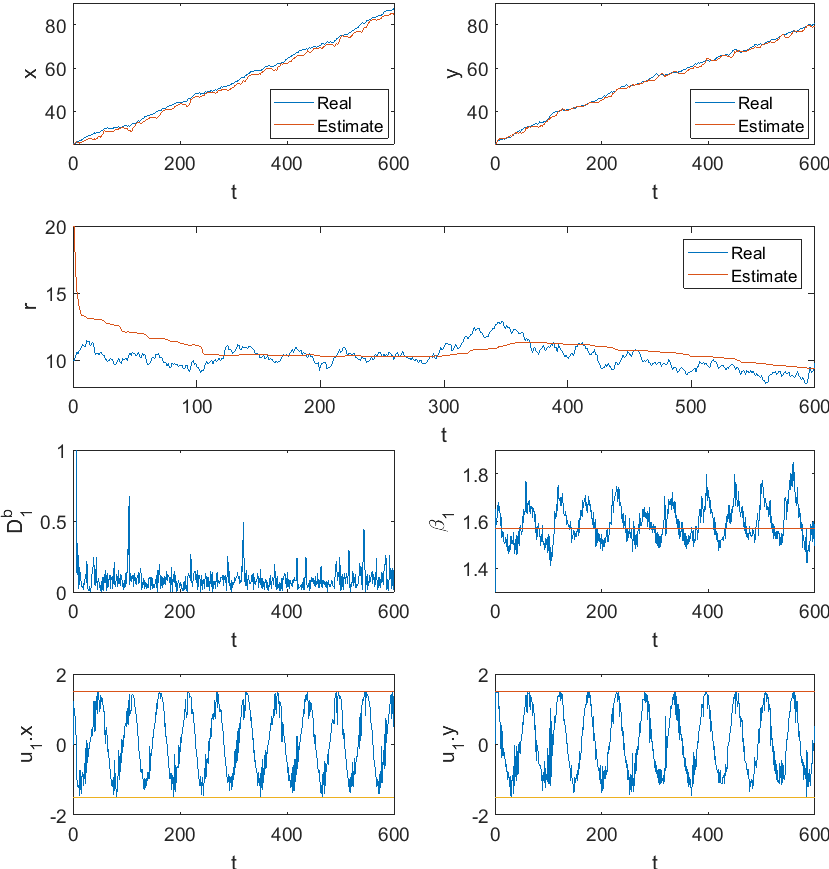}
    \caption{First and second row: real and estimated target's centre $c:x,y$ and radius $r$. Third row: tracking error of agent A1, $D_1^b$ and angle $\beta_1$. Fourth row: control input of agent A1, $u_1:x,y$}
    \label{targets}
\end{figure}

%%%%%%%%%%%%%%%%%%%%%%%%%%%%%%%%%%%%%%%%%%%%%%%%%%%%%%%%%%%%%%%%%%%%%
\section{CONCLUSIONS}\label{s: conclusion}

In this paper, we considered the circumnavigation problem of moving circles with varying radius. Our solution relied on $n$ agents for circumnavigation but only one for measurements as well as on a satellite with noisy and sparse estimates of the target. The measurements taken were the distance of the sensing agent to the boundary of the target as well as to its centre. A control protocols was proposed and its convergence to the desired behaviour was proved for given conditions such as $\dot p_1(t)$ and $\dot D_1^c(t)$ being p.e.. 

Future work includes the circumnavigation of this same irregular shape but without assuming it can be approximated by a moving circle. Also, having this moving circle, we would like to explore how we can do the circumnavigation with access to only the distance to the target boundary. This might be achieved exploiting the number of agents available and assuming they are all able to measure this distance. 

\addtolength{\textheight}{-12cm}   % This command serves to balance the column lengths

%%%%%%%%%%%%%%%%%%%%%%%%%%%%%%%%%%%%%%%%%%%%%%%%%%%%%%%%%%%%%%%%%%%%%

%\begin{thebibliography}{99}

%\bibitem{iman} Iman Shames, Soura Dasgupta, Fellow, Barıs¸ Fidan, and Brian D. O. Anderson, Circumnavigation Using Distance Measurements Under Slow Drift.
%\bibitem{johanna} Johanna O. Swartling, Iman Shames, Karl H. Johansson, Dimos V. Dimarogonas, Collective Circumnavigation.
%\bibitem{jason} Jason Kong, Mark Pfeiffer, Georg Schildbach, Francesco Borrelli, Kinematic and Dynamic Vehicle Models for Autonomous Driving
%Control Design

%\end{thebibliography}

\bibliographystyle{IEEEtran}
\bibliography{references}

\end{document}